\documentclass{amsart}
\usepackage{amsmath,amsfonts} 
\usepackage{amsthm} 
\usepackage{amssymb} 
\usepackage[T1]{fontenc} 
\usepackage[utf8]{inputenc}
\usepackage{textcomp}
\usepackage[dvips]{graphicx}
\usepackage{verbatim}
\usepackage{todonotes}
\usepackage{mathrsfs}
\usepackage[english]{babel}
\usepackage[babel]{csquotes}
\usepackage{xy}
\xyoption{all}
\usepackage{hyperref}
\usepackage[style=alphabetic,hyperref,backend=bibtex]{biblatex}
\bibliography{bibliography}

\theoremstyle{definition}
\newtheorem{definition}{Definition}[section]
\theoremstyle{plain}
\newtheorem{theorem}[definition]{Theorem}
\newtheorem*{theoremno}{Theorem}
\newtheorem{proposition}[definition]{Proposition}
\newtheorem{lemma}[definition]{Lemma}
\newtheorem{corollary}[definition]{Corollary}
\theoremstyle{remark}
\newtheorem{remark}[definition]{Remark}
\newtheorem{example}[definition]{Example}

\newtheorem*{question}{Question}
\DeclareMathOperator{\EFin}{EFin}
\DeclareMathOperator{\Ext}{Ext}
\DeclareMathOperator{\Hom}{Hom}
\DeclareMathOperator{\coker}{coker}
\DeclareMathOperator{\Spec}{Spec}

\DeclareMathOperator{\Gal}{Gal}
\DeclareMathOperator{\GL}{GL}

\DeclareMathOperator{\Aut}{Aut}

\DeclareMathOperator{\Repf}{Repf}
\DeclareMathOperator{\Mod}{Mod}
\DeclareMathOperator{\Modf}{Modf}

\DeclareMathOperator{\Stab}{Stab}
\DeclareMathOperator{\HG}{HG}
\DeclareMathOperator{\ob}{Ob}

\DeclareMathOperator{\Bun}{Bun}
\DeclareMathOperator{\Vtf}{Vtf}

\DeclareMathOperator{\SL}{SL}

\newcommand{\fr}{\mathsf{for}}
\newcommand{\aut}{\underline{\Aut}^\otimes}
\newcommand{\E}{\mathbb{E}}
\newcommand{\F}{\mathbb{F}}

\newcommand{\ebar}{{\bar\eta}}
\newcommand{\Oh}{\mathscr{O}}

\newcommand{\N}{\mathbb{N}}

\newcommand{\D}{\mathscr{D}}

\newcommand{\de}{\partial}

\newcommand{\A}{\mathbb{A}}
\newcommand{\R}{\mathbb{R}}

\newcommand{\bx}{{\bar{x}}}

\begin{document}

\title[The base change of the monodromy group]{The base change of the monodromy group for geometric Tannakian pairs}
\author{Giulia Battiston}
\date{\today}
\address{ Ruprecht-Karls-Universität Heidelberg -Mathematisches Institut -
Im Neuenheimer Feld 288
D-69120 Heidelberg }
\thanks{This work was supported by the MAthematics Center Heidelberg (MATCH)}
\subjclass[2010]{12F15, 19D23, 14L15}
\email{gbattiston@mathi.uni-heidelberg.de}

\begin{abstract}We prove that in any characteristic the formation of the monodromy group of a $\D$-module commutes with the extension of the ground field, extending a result of Gabber for separable extensions.
\end{abstract}
\maketitle

If $S$ is a smooth geometrically connected variety over a field of characteristic zero $K$, after choosing $s\in S(K)\neq\emptyset$, one can associate to every finite rank vector bundle $E$ endowed with a flat connection $\nabla$ an algebraic group over $K$, the \emph{monodromy group} $\pi((E,\nabla),\omega_s)$ (also called the differential Galois group in \cite{Ka:co} or holonomy group in \cite{Sa:ct}), which is constructed via Tannaka duality from the full Tannakian category spanned by $(E,\nabla)$ in the category of all flat connections on $S$.

If $L/K$ is a field extension, there are two kind of base change that one can perform on $\pi((E,\nabla),\omega_s)$: on one hand one can simply consider the base change of the monodromy group to $L$. On the other, one can first base change $(E,\nabla)$ to a vector bundle wit flat connection $(E_L,\nabla_L)$ on $S_L=S\times_{\Spec K}\Spec L$ and then construct the monodromy group of the latter. 

It is a result of Gabber (\cite[Prop.~1.3.2]{Ka:ca}) that the two construction lead to the same object, namely that the formation of the monodromy group commutes with the extension of the ground field. The proof of this result is very nice and simple, relying only on a description of such a group as stabilizer in a certain invertible matrix group of a family of subspaces that are strictly related to the categories spanned by $(E,\nabla)$ and $(E_L,\nabla_L)$ respectively, together with Galois descent. 

The key point, for our purposes, is that Galois descent is always linked to the action of the Galois group $\Gal(L/K)$ of $L$ over $K$ through the \emph{automorphisms} of some object $X$ that we want to descend. This is important because it allows to make use of the following crucial trick: if $V$ is a finite rank $L$-vector space on which $\Gal(L/K)$ acts, then for any $\alpha\in\Gal(L/K)$ and $W\subset V$ sub-vector space, one has that $\alpha(W)$ is again a $L$-vector space (as $\lambda\cdot\alpha(w)=\alpha(\alpha^{-1}(\lambda)\cdot w)$ for every $\lambda\in L$ and $w\in W$) and that
\[\Stab_{\GL(V)}(\alpha(W))=\alpha(\Stab_{GL(V)}(W)).\]

If now $K$ has positive characteristic, the very same question makes sense when working, instead than with finite rank vector bundles with flat connections, with $\D$-modules of finite rank. If $L/K$ is a separable extension the very same argument of Gabber applies, but if $L/K$ is inseparable, it is clearly impossible to use Galois descent. 
There are descent theory for (purely) inseparable field extensions (for example \cite{AS}) but they rely on the action through the \emph{endomorphism} of the object we want to descend, hence they do not allow to make use of the major trick we just described.

Instead, we chose to use the descent theory described in \cite{bat}, namely for $L/K$ finite and modular (see Definition~\ref{modular}) the descent of an algebraic object $X$ from $L$ to $K$ is related to the action of $\HG(L/K)=\Aut(L[\bx]/K[\bx])$ on $X\otimes_LL[\bx]$, where $L[\bx]=L[x]/x^{p^e}$ for suitable $e$ (see Theorem~\ref{ins}). In this way we gain an action through automorphisms but we lose of course some of the strong properties that come from working over a field. It turns out, though, that the local Artin algebra $L[\bx]$ still enjoys many good properties that allow to complete the proof following the ideas of Gabber also in the inseparable case, hence proving the main result of this article, namely
\begin{theoremno}[{see Corollary~\ref{strat}}]
Let $K$ be a field of positive characteristic and $S$ be a smooth geometrically connected variety over $K$, with $S(K)\neq\emptyset$, then for every $\D$-module of finite rank over $S$ the formation of the monodromy group commutes with the extension of the ground field.
\end{theoremno}

Let us now describe the contents of this article: Section~\ref{one} is dedicated to a quick review of the theory of Tannaka categories, and to Katz's explicit description of the monodromy group in term of stabilizers. In Section~\ref{geom:pairs} we describe the proof of Gabber in the separable case, under a more general axiomatization that applies to many other Tannakian categories (see Ex.~\ref{ex1}), in Section~\ref{three} we define a finer axiomatization in order to prove the same theorem for $L/K$ finite and inseparable, and prove that the category of $\D$-modules over a smooth geometrically connected variety satisfies this new axiomatization. In Section~\ref{four} we give applications of the main theorem, for example completing the missing case of a theorem of Esnault and Langer (see Theorem~\ref{elgen}) and proving that the field extension morphism for many Tannakian fundamental groups is faithfully flat (see Lemma~\ref{ff}).

\subsection*{Notation} If $K$ is a field, $\Vtf_K$ denotes the category of finite dimensional $K$-vector spaces. More in general if $S$ is a scheme  $\Bun_S$ denotes the full subcategory of $\Oh_S$-coherent modules whose objects are locally free modules of finite rank. If $S=\Spec (A)$ we use $\Bun_A$ instead of $\Bun_{\Spec(A)}$. The category of modules (respectively, of finite type modules) over some ring $A$ will be denoted as $\Mod_A$ (respectively, $\Modf_A$).

Let $X$ be an object of a Tannakian category $\mathcal{T}$, then by a \emph{linear algebra construction} $\mathcal{C}(V)$ over $V$ we mean an object constructed from $V$ via a finite sequence of tensor products, direct sums and duals.

If $X$ is an object (scheme, module, and so on) defined over an algebra $A$ (or over its spectrum) and $B$ is an $A$-algebra, we will denote by $X_B$ its base change to $B$ (or to $\Spec B$).

\section{Neutral Tannakian categories}\label{one}
Let $K$ be a field of characteristic $0$ and let $S$ be a smooth $K$-variety such that $S(K)\neq \emptyset$. Let $\D\Mod(S/K)$ be the category of finite rank vector bundles over $S$ endowed with a flat connection. Let us fix $s\in S(K)$ and define
\[\omega_s:\D\Mod(S/K)\to \Vtf_K\]
by mapping $(E,\nabla)\in\D\Mod(S/K)$ to the fibre of $E$ at $s$.
It is a classical result (see \cite[VI, Sec.~1.2]{Sa:ct}) that  $(\D\Mod(S/K),\omega_s)$ is a \emph{Tannakian category} over $K$. We will not give here the complete definition of such a category, referring the reader to \cite{Sa:ct} and \cite{DM}, but note that if $K$ is a field and $(\mathcal{T}, \omega:\mathcal{T}\to  \Vtf_K)$ is a neutral Tannakian category over $K$, one has that  is in particular $\mathcal{T}$ is abelian, $K$-linear, it admits a tensor structure and every object has a dual. Moreover, if $X$ is an object in $\mathcal{T}$, there is a unique smallest full sub-Tannakian category of $\mathcal{T}$ containing $X$, that we denote by $\langle X\rangle_\otimes$. As for $\omega$, it is a $K$-linear, faithful and a tensor functor, that is, it respects the tensor structure on $\mathcal{T}$.

The main theorem in the theory of neutral Tannakian categories is the following:
\begin{theorem}[{Tannakian duality, \cite[Thm.~2.11]{DM}}]\label{tannakaduality}
Let $K$ be a field and $(\mathcal{T},\omega)$ a neutral Tannakian category over $K$. Let $\aut(\omega)$ be defined on every $K$-algebra $i:K\to A$ as
\[\aut(\omega)(A)=\{i^*\circ\omega\simeq i^*\circ\omega\text{ isomorphism of tensor functors}\},\]
then $\aut(\omega)$ is representable by an affine $K$-group scheme denoted by $\pi(\mathcal{T},\omega)$ and 
\[F_\omega:(\mathcal{T},\omega)\to(\Repf_K\pi(\mathcal{T},\omega),\fr)\]
induced by $\omega$ is an equivalence of neutral tensor categories, where $\fr$ is the forgetful functor.
\end{theorem}

As remarked in \cite{Ka:ca}, if $X\in(\mathcal{T},\omega)$ and we consider the neutral Tannakian category $(\langle X\rangle_\otimes, \omega)$, where by abuse of notation we still denote by $\omega$ its restriction to $\langle X\rangle_\otimes$, then the $K$-algebraic group $\pi(X,\omega)=\pi(\langle X\rangle_\otimes,\omega)$, called the \emph{monodromy group of $X$}, admits a simpler description, thanks to the following two theorems:

\begin{theorem}[Chevalley] Let $K$ be a field and let $G$ an affine group scheme of finite type over $K$. Let $H$ be any subgroup of $G$, then there exists a representation $V$ of $G$ and a subspace $W\subset V$ such that $H$ is the stabilizer of $W$ in $V$.
\end{theorem}

\begin{theorem}[{\cite[3.5, Thm.]{Wa}}]
Let $G$ be an algebraic group over a field $K$ and $V$ be a faithful representation of $G$, then every finite dimensional representation of $G$ is a sub-quotient of some linear construction $\mathcal{C}(V)$ with the natural induced $G$-action.
\end{theorem}

In particular, as remarked by Katz in \cite[III]{Ka:co}, if $G$ is a subgroup of $GL(V)$ for some $V$, then $G$ is uniquely determined by the list of subspaces $\mathcal{S}_G$ it stabilizes in all linear algebra construction $\mathcal{C}(V)$ over $V$.

Let $X\in\mathcal{T}$, then $\pi(X,\omega)$ is a subgroup of $\GL(\omega(X))$, and by Tannaka duality, it stabilizes all and only the sub-objects of $\mathcal{C}(\omega(X))=\omega (\mathcal{C}(X))$ of the form $\omega(Y)$ for some $Y\in\langle X\rangle_\otimes$, where $\mathcal{C}$ runs along all the possible linear algebra constructions over $\omega(X)$. In particular we have the following:
\begin{lemma}[{\cite[Sec.~III]{Ka:co}}]\label{description}
Let $X$ be an object in a neutral Tannakian category $(\mathcal{T}, \omega)$ over a field $K$. Then 
\[\pi(X,\omega)=\Stab_{\GL(\omega(X))}\{W\mid \exists\,\mathcal{C},\exists\,  Y\subset\mathcal{C}(X)\text{ such that } W=\omega(Y)\},\]
where $\mathcal{C}$ runs along all possible linear algebra constructions.
\end{lemma}

\section{Galois Tannakian pairs and the base change of the monodromy group}\label{geom:pairs}

Keeping the notations of the previous section, if $(E,\nabla)\in\D\Mod(S/K)$ and $L/K$ is a field extension, we can consider the vector bundle with flat connection $(E_L,\nabla_L)\in \D\Mod(S_L/L)$. If we denote again by $\omega_s$ the fibre functor given by $s\in S(K)\subset S(L)=S_L(L)$, from the definition of $\pi((E_L,\nabla_L),\omega_s)$ it follows that there is an inclusion of $L$-group schemes
\[\mu:\pi((E_L,\nabla_L), \omega_s)\subset\pi((E,\nabla),\omega_s)_L.\]
It is natural to ask whether $\mu$ is an isomorphism, and it was proven by Gabber (see \cite[Prop.~1.3.2]{Ka:ca})  that this is indeed the case. We will recall his proof in the following more generic setting:

\begin{definition}
 Let $K$ be any field and $i:K\subset L$ a field extension. A pair consisting of a neutral Tannakian category $(\mathcal{T},\omega)$ over $K$ and a neutral Tannakian category  $(\mathcal{T}_L,\omega_L)$ over $L$ is a \emph{geometric $(F,L)$-pair} if there exists an additive, exact tensor functor $\centerdot_L:\mathcal{T}\to\mathcal{T}_L$, $A\mapsto A_L$ called the \emph{base change} functor such that
\begin{itemize}
\item[(A1)] one has that  $\omega_L\circ \centerdot_L= i^*\circ\omega$.
\end{itemize}
A geometric $(K,L)$-pair is said to be \emph{Galois} if
\begin{itemize}
\item[(B1)] $K$ is the field of fixed elements of $G=\Aut(L/K)$;
\item[(B2)] for every $X\in\mathcal{T}$, $Y\subset X_L$ and $\sigma\in G$ there exists $Y'\subset X_L$ such that $\omega_L(Y')=\sigma (\omega_L(Y))$ where $G$ acts naturally on $\omega_L(X_L)=\omega(X)\otimes_KL$.
\item[(B3)] let $X\in\mathcal{T}$, then if $Y\subset X_L\in\mathcal{T}_L$ and $V\subset \omega(X)\in \Vtf_K$ are such that $\omega_L(Y\subset X_L)=(V\subset \omega(X))_L$,  there exists $Z\subset X\in\mathcal{T}$ such that $(Z\subset X)_L=Y\subset X_L$.
\end{itemize}
\end{definition}

\begin{example}\label{ex1} Let $L/K$ satisfying $(B1)$, then the following are examples of Galois $(K,L)$-geometric pairs, where the axiom $(B2)$ is given by the existence of a Galois action on $\mathcal{T}_L$ commuting with the fibre functor $\omega_L$:
\begin{itemize}
\item[i)] Consider the category $\D\Mod(S/K)$ of $\D$-modules of finite rank $(E,\nabla)$ on a smooth geometrically connected variety $S$ over $K$ such that $S(K)\neq\emptyset$, together with the functor $\omega_s(E,\nabla)=E_s$ (see \cite[Chap.~VI,1.2]{Sa:ct} together with \cite[§16]{SGA4}). Then $(\D\Mod(S/K),\omega_s)$ and $(\D\Mod(S_L/L),\omega_s)$, where $s$ is considered in $S_L(L)=S(L)$ in the second case, is a Galois pair.	
\item[ii)] More in general, for any $S$ locally of finite type over $K$ with $S(K)\neq 0$ if one can consider the category of stratified bundles (see again \cite[Chap.~VI,1.2]{Sa:ct})) over $S$ and $S_L$ and the fibre functor as in the previous example, they form a Galois pair.
\item[iii)] The category $\EFin S$ of essentially finite sheaves over a pseudo-proper geometrically connected and geometrically integral $F$-variety $S$ (see \cite[Def.~7.7 and 7.1,Prop.~5.5]{vist}) such that $S(K)\neq \emptyset$ together with $\EFin S_L$ and fibre functors as in $(ii)$.
\item[iv)] Variations of the previous examples, such as the subcategories of unipotent objects, finite tame objects (see \cite[Def.~12.1]{vist}) or essentially finite objects (see \cite[Def.~7.7,Cor~7.10]{vist}), the largest semi-simple sub-category, and so on.
\end{itemize}
The axiom $(B3)$ follows by (a generalization of) an argument of Gabber in \cite[Prop.~1.3.2]{Ka:ca}. Namely, if $Y\subset X_L$ and if 
\[\omega_L(Y\subset X_L)=(V\subset \omega(X))_L\]
for some $V\subset \omega(X)\in \Vtf_K$, then $Y$ is invariant under the natural action of $\Gal(L/K)$ on $X_L$: as $\omega_L$ is exact, it preserves pullbacks. Moreover as the Galois action commutes with $\omega_L$ and $V_L$ is $\Gal(L/K)$-invariant,
\[\omega_L(Y\times_{X_L} \sigma(Y))=V_L\times_{\omega(X)_L}\sigma(V_L)=V_L\cap\sigma(V_L)=V_L,\]
and as $\omega_L$ is faithful and $\omega_L(\coker (\sigma(Y)\times_{X_L} Y\to Y))$ is zero, it follows that $\coker (\sigma(Y)\times_{X_L} Y\to Y)$ is also zero and hence
\[\sigma(Y)\times_{X_L} Y\simeq Y\]
as sub-objects of $X_L$. The same holds for $\sigma(Y)$ and hence $Y\simeq\sigma(Y)$ as sub-objects of $X_L$, that $Y$ is $\Gal(L/K)$-invariant. 
In particular, as Galois descent holds on all categories of Example~\ref{ex1}, there exists $Z\subset X\in\mathcal{T}$ such that $Y\subset X_L=(Z\subset X)_L$.
\end{example}

Then the following holds:
\begin{theorem}[{\cite[Prop.~1.3.2]{Ka:ca}}]\label{perfect}
Let $(\mathcal{T},\omega)$ and $(\mathcal{T}_L,\omega_L)$ be a Galois geometric $(K,L)$-pair and let $X\in\ob\mathcal{T}$. Then there is a functorial isomorphism 
\[\mu:\pi((X_L,\omega_L))\to\pi(X)_L.\]
\end{theorem}
\begin{proof}
Let us denote $V_L=\omega_L(X_L)=\omega(X)_L$, by Lemma~\ref{description} one has that $\pi(X_L,\omega_L)\subset GL(V_L)$ is exactly the stabilizer of all $W\subset \mathcal{C}(V_L)$ such that $W=\omega_L(Y)$ for some $Y\subset \mathcal{C}(X_L)$, where $\mathcal{C}$ is some linear algebra construction. Note that $\mathcal{C}(X_L)=\mathcal{C}(X)\otimes_KL$, hence there is a natural action of $\Gal(L/K)$ over $\mathcal{C}(X_L)$ given by $\sigma\mapsto id\otimes \sigma$. Let now $Y\in \langle X_L\rangle_\otimes$ be a sub-object of $\mathcal{C}(X_L)$, then $\sigma(\omega(Y))$ is again in the strict image of $\omega$ by $(B2)$.
In particular, the set of subspaces $W$ of $\mathcal{C}(\omega_L(X_L))$ such that $W=\omega_L(Y)$ for some $Y\in\langle X_L\rangle_\otimes$ is invariant under the action of $\Gal(L/K)$, hence so are the equations defining $\pi(X_L,\omega_L)$ in $GL(V_L)$. 

By Galois descent, it follows that $\pi(X_L,\omega_L)$ is defined over $K$, that is, that there exists $G\subset GL(\omega(X))$ such that $\pi(X_L,\omega_L)=G\otimes_K L$. As $\mu$ is a closed immersion, it suffices to prove that $\pi(X,\omega)\subseteq G$.
By Chevalley theorem, there exists $\mathcal{C}_G$ a linear algebra construction and $W\subset \mathcal{C}_G(\omega(X))$ such that $G=\Stab(W)$. In order to prove that $\pi(X,\omega)\subseteq G$, it is hence enough to prove that $\pi(X,\omega)$ stabilizes $W$.

As $\pi(X_L,\omega_L)=G\otimes_K L$, we know that $W_L$ is stabilized by $\pi(X_L,\omega_L)$, in particular there exists $Y\subset \mathcal{C}_G(X_L)$ such that $W_L=\omega_L(Y)$. But then by axiom $(B3)$ there exists $Z\in\mathcal{T}$ such that $Y=Z_L$, in particular $\pi(X,\omega)$ stabilizes $W=\omega(Z)$ and this completes the proof.
\end{proof}

\section{Modular field extensions and modular pairs}\label{three}

The goal of this section is to prove a result equivalent to Theorem~\ref{perfect} for any field extension $L/K$, also in case the axiom $(B1)$ is not satisfied. The issue here is of course that Galois descent does not work in this situation. There is though a way to bypass this problem, using the theory of descent along modular field extensions of finite exponent.

\begin{definition}[{ see \cite[Thm.~1]{Swee} and \cite[Def~1.1]{devmo}}]\label{modular}
An algebraic extension $L$ of $K$, is said to be of \emph{finite exponent} if there is a positive natural number $e$ such that $L^{p^e}$ is separable over $K$, where $p$ is the characteristic of $K$. The minimal of such $e$ is called the \emph{exponent} of $L$ over $K$.

An algebraic extension $L$ of $K$ is \emph{modular} if it is isomorphic to a (possibly infinite) tensor product over $K$ of elementary extensions (that is, of extensions of $K$ generated by one element).
\end{definition}

In case $L/K$ is modular of finite exponent $e$ we define (see \cite[Sec.~1]{bat} and \cite{Hee}) the \emph{Heerma--Galois group of $L$ over $K$} to be 
\[\HG(L/K)=\Aut(L[\bx]/K[\bx])\]
where $K[\bx]=K[x]/x^{p^e}$, and similarly for $L[\bx]$. Then we have the following:

\begin{proposition}[{\cite[Thm.~2.11]{bat}}]\label{ins}
Let $L/K$ be a finite normal modular extension of exponent less or equal than $e$ and let us denote $K[\bx]=K[x]/(x^{p^e})$.  Let $V$ be a $K$-vector space and $W$ a sub-$K$-vector space of of $V\otimes_KL$. Then the following are equivalent:
\begin{itemize}
\item[i)] there exists $W_0\subset V$ over $K$ such that $W=W_0\otimes_KL\subset V\otimes_KL$;
\item[ii)] $W\otimes_LL[\bx]$ is invariant under the natural action of $\HG(L/K)$ on $V\otimes_KK[\bx]$.
\end{itemize}
\end{proposition}

The previous theorem is not enough to generalize Gabber's proof, as for $L[\bx]$-linear categories Tannaka duality requires some additional conditions to be fulfilled:

\begin{theorem}[{\cite[Lemma~8.1.2]{And}\footnote{This theorem is a more manageable version of \cite[II, Prop.~3.1.4.3]{Sa:ct} and a more explicit version than \cite[Thm.~1.3.4]{scha}.}}]\label{ta:gen}
Let $A$ be a commutative ring and $\mathcal{T}_A$ a $A$-linear category, let $\omega_A:\mathcal{T}_A\to \Modf_{A}$ a tensor functor which is exact and faithful. Suppose that there is a tensor subcategory $\mathcal{T}_A'$ such that $\mathcal{T}_A'$ is rigid \footnote{ We refer to \cite{Sa:ct} for the notion of rigid category and  rigid functor.}, $\omega_A$ restricted to $\mathcal{T}_A'$ is a rigid functor with image in $\Bun_{A}$ and that $\mathcal{T}'_A$ \emph{spans} $\mathcal{T}_A$, that is every object of $\mathcal{T}_A$ is a quotient of some object in $\mathcal{T}_A'$. Then there exists a \emph{flat} $A$-group scheme $G$ representing $\aut(\omega_A)$, moreover $\omega_A$ induces an equivalence of categories making the following diagram commute:
\[\xymatrix{\mathcal{T}_A\ar[rd]_{\omega_A}\ar[rr]^{\simeq} && \Repf_{A}(G)\ar[dl]^{\fr}\\
& \Modf_{A}
}.\]
\end{theorem}

In view of the previous theorem we generalize the notion of Galois pair to the following:
\begin{definition}
A geometric $(K,L)$-pair $(\mathcal{T},\omega)$, $(\mathcal{T}_L,\omega_L)$ is said to be \emph{$e$-modular} if
\begin{itemize}
\item[(C1)] $L/K$ is finite modular of exponent $e$ and there exists a $L[\bx]$-linear categories $\mathcal{T}_{L[\bx]}'\subset\mathcal{T}_{L[\bx]}$ together with a tensor functor 
\[\omega_{L[\bx]}:\mathcal{T}_{L[\bx]}\to \Modf_{L[\bx]}\]
satisfying all hypothesis of Theorem~\ref{ta:gen};
\item[(C2)]there exists an additive, exact tensor functor $\centerdot_{L[\bx]}:\mathcal{T}_L\to\mathcal{T}'_{L[\bx]}$, $A\mapsto A_{L[\bx]}$ such that if $i:L\hookrightarrow L[\bx]$ is the natural inclusion one has that 
\[\omega_{L[\bx]}\circ \centerdot_{L[\bx]}= i^*\circ\omega_L\]
and moreover for every $X\in\mathcal{T}$, $Y\subset (X_L)_{L[\bx]}$ and $\sigma\in \HG(L/K)$ there exists $Y'\in\mathcal{T}_{L[\bx]}$, $Y'\subset X_{L[\bx]}$ such that $\omega_{L[\bx]}(Y')=\sigma (\omega_{L[\bx]}(Y))$ where $\HG(L/K)$ acts naturally on $\omega_{L[\bx]}((X_L)_{L[\bx]})=\omega(X)\otimes_KL[\bx]$.
\item[(C3)] let $X\in\mathcal{T}$, then if $(Y\subset X_L)\in\mathcal{T}_L$ and $V\subset X\in \Vtf_K$ are such that $\omega_L(Y\subset X_L)=(V\subset \omega(X))_L$,  there exists $Z\subset X\in\mathcal{T}$ such that $(Z\subset X)_L=Y\subset X_L$.
\item[(C4)] there exists an additive tensor functor $\centerdot_{_0}:\mathcal{T}_{L[\bx]}\to\mathcal{T}_L$ such that 
\[\omega_L\circ\centerdot_{0}=\omega_{L[\bx]}(\centerdot)\otimes_{L[\bx]}L.\]
\end{itemize}
\end{definition}

Before proving the main result of this section, let us recollect some properties of the category $\Mod_{L[\bx]}$:
\begin{lemma}
Let $M$ be an $L[\bx]$-module, then the following are equivalent:
\begin{itemize}\label{a}
\item[i)] $M$ is free;
\item[ii)]  $M$ is projective;
\item[iii)] $M$ is flat;
\item[iv)] $M$ is injective.
\end{itemize}
\end{lemma}
\begin{proof}
It is classical that $(i)\Rightarrow (ii)\Rightarrow (iii)$, and the reverse arrows are given by the fact that $L[\bx]$ is a local Artin ring.
Moreover, $L[\bx]$ is a Frobenius ring (see \cite[Ex.~3.15B]{lam})  in particular a module is projective if and only if it is injective (\cite[Thm.~15.9]{lam}). 
\end{proof}

\begin{theorem}
Let $(\mathcal{T},\omega)$ and $(\mathcal{T}_L,\omega_L)$ be a geometric $e$-modular $(K,L)$-pair. Let $X\in\mathcal{T}$ such that $\mathcal{T}_{L[\bx]}'\cap\langle X_{L[\bx]}\rangle_\otimes$ spans $\langle X_{L[\bx]}\rangle_\otimes$\footnote{As in Theorem~\ref{ta:gen}, a subcategory $\mathcal{T}'\subset\mathcal{T}$ \emph{spans} $\mathcal{T}$ if all objects of the latter are quotients of of objects in $\mathcal{T}'$.}. Then the natural inclusion
\[\mu:\pi(X_L,\omega_L)\to \pi(X,\omega)_L\]
is an isomorphism.
\end{theorem}
\begin{proof} 
We follow the ideas of Theorem~\ref{perfect}: first of all we want to prove that $\pi(X_L,\omega_L)$ is defined over $K$, which by Proposition~\ref{ins} is equivalent to show that $\pi(X_L,\omega_L)_{L[\bx]}\subset \omega(X)\otimes_KL[\bx]$ is $\HG(L/K)$-invariant.
Let us shorten $(X_L)_{L[\bx]}$ by simply writing $X_{L[\bx]}$. Let us consider the set
\[\mathcal{S}=\{Y\in\mathcal{T}_{L[\bx]}'\mid Y\subset\mathcal{C}(X_{L[\bx]})\}\]
where $\mathcal{C}$ varies along all linear algebra constructions, and let us consider the functor $\Stab_{\GL(\omega(X)_{L[\bx]})}(\mathcal{S})$ (note that by the axioms $(A1)$ and $(C2)$ we have that $\omega_{L[\bx]}(X_{L[\bx]})=\omega(X)_{L[\bx]}$). For every $Y\in\mathcal{S}$ we have that $\omega_{L[\bx]}(Y)$ is projective (hence injective by Lemma~\ref{a}) on $L[\bx]$, in particular $\mathcal{C}(\omega(X)_{L[\bx]})/\omega_{L[\bx]}(Y)$ is projective as well and, by \cite[I, Sec~2.12,(5)]{Jan} $\Stab_{GL(\omega(X)_{L[\bx]})}(Y)$ is representable by a closed subgroup of $GL(\omega(X)_{L[\bx]})$. In particular $\Stab_{\GL(\omega(X)_{L[\bx]})}(\mathcal{S})$ is also representable by a closed subgroup of $\GL(\omega(X)_{L[\bx]})$ that we denote $H$: it is a $\HG(L/K)$-invariant subgroup of $\GL(\omega(X)_{L[\bx]})$ (by $(C3)$) contained in $\pi(X_L,\omega_L)_{L[\bx]}$.

In order to prove that they are actually equal, and hence that the latter is $\HG(L/K)$-invariant, we need an additional step. By Theorem~\ref{ta:gen}, there exists a flat group scheme $G$ representing $\aut(\omega_{L[\bx]})$ (where we consider $\omega_{L[\bx]}$ restricted to $\langle X_{L[\bx]}\rangle_\otimes$),  in particular for every $L[\bx]$-algebra $a:L[\bx]\to R$ there is a natural monomorphism
\[j_R:G(R)=\aut(a^*\circ\omega_L)\to H(R)=\Stab_{GL(\omega(X)_{L[\bx]})\otimes_{L[\bx]}R)}(\mathcal{S}\otimes_{L[\bx]}R)\] 
sending $\alpha\in G(R)$ to its value $\alpha_{X_{L[\bx]}}$ on $X_{L[\bx]}$. As $L[\bx]$ is Artinian, $j:G\to H$ is a closed immersion (see  \cite[$VI_A$, Prop.~2.5.2 (c)]{SGA3}).

We have hence the following chain of closed subgroup schemes over $\Spec L[\bx]$:
\[G\subset H\subset \pi(X_L,\omega_L)_{L[\bx]}\subset GL(\omega(X))_{L[\bx]})\]

In order to conclude let us remark the following: both $G$ and $\pi(X_L,\omega_L)_{L[\bx]}$  have free defining ideals $\mathcal{I}_{\pi(X_L,\omega_L)}\otimes_LL[\bx]\subset \mathcal{I}_G$ over $L[\bx]$: for $\pi(X_L,\omega_L)_{L[\bx]}$ it is clear, while $G$ is flat, hence projective (Lemma~\ref{a}), so $\mathcal{I}_G$ is a direct sum of the ring of global sections of $\GL(\omega(X)_{L[\bx]})$, hence projective as well. In particular as the short exact sequence
\begin{equation}\label{ses}
\mathcal{I}_{\pi(X_L,\omega_L)}\otimes_LL[\bx]\subset \mathcal{I}_G\to Q
\end{equation}
splits, the cokernel $Q$ is a direct summand of a free module hence it is projective (and free).
Let $G_0$ be the closed fiber of $G$, in order to conclude that $\pi(X_L,\omega_L)\otimes_LL[\bx]$ is $\HG(L/K)$-invariant is it is enough to show that it is equal to $G$, which is equivalent to show that the inclusion $G_0\subset\pi(X_L,\omega_L)$ is an equality as the  split exact sequence \eqref{ses} remains exact after taking its closed fiber.

Now if $a:L\to R$ is an $L$-algebra, by definition $G_0(R)=G(R)$ where on the right hand side $R$ is considered as an $L[\bx]$-algebra through the projection $L[\bx]\to L$. 
Let now $Y\subset X_{L[\bx]}$, if $Y\in\mathcal{T}'_{L[\bx]}$ and $Q$ is the cokernel of this inclusion, we have that the short exact sequence
\[\omega_{L[\bx]}(Y)\to \omega_{L[\bx]}(X_{L[\bx]})\to\omega_{L[\bx]}(Q)\]
splits, in particular the last term is projective, hence free, and thus by $(C4)$ one has that $\omega_L(Y_0)\subset\omega_L(X)$, which by exactness implies $Y_0\subset X$, in particular $Y_0\in\langle X\rangle_\otimes$. If $Y$ is not in $\mathcal{T}'_{L[\bx]}$ then by $(C1)$ it is a quotient of $Y'\in\mathcal{T}'_{L[\bx]}$, thus $\omega(Y_0)$ is a quotient of $\omega(Y'_0)$ and by exactness $Y_0$ is a quotient of $Y'_0$ in particular it is in $\langle X\rangle_\otimes$. This, together with Lemma~\ref{description}, implies that $\pi(X_L,\omega_L)(R)\subset G_0(R)$ which completes the argument.

The rest of the proof goes exactly as in Theorem~\ref{perfect}, using $(C3)$ instead of $(B3)$.
\end{proof}

\begin{proposition}\label{dmod}
Let $S$ be a smooth geometrically connected scheme over a field $K$ of positive characteristic, with $S(K)\neq\emptyset$. Let $s\in S(K)$ and $L/K$ be a finite modular field extension of exponent $e$. Then $(\D\Mod(S/K),\omega_s)$ and $(\D\Mod(S_L/L),\omega_s)$ are an $e$-modular geometric pair and for every $\E=(E,\nabla)\in\D\Mod(S/K)$ the natural inclusion
\[\mu:\pi(\E_L,\omega_s)\subset \pi(\E,\omega_s)_L\]
is an isomorphism.
\end{proposition}
\begin{proof}
Let us first prove the existence of the category of axiom $(C1)$. We take $\mathcal{T}_{L[\bx]}=\D\Mod(S_{L[\bx]}/L[\bx])$ to be the category of $\Oh_{S_{L[\bx]}}$-coherent modules endowed with a $\D$-module structure, the rigid subcategory $\mathcal{T}_{L[\bx]}'$ to be given by locally free objects and the fiber functor 
\[\omega_{L[\bx]}:\D\Mod(S_{L[\bx]}/L[\bx])\to \Mod_{L[\bx]}\]
to be the base change of the section $s\in S(K)$. Namely, $\omega_{L[\bx]}(E,\nabla)=E\otimes_{\Oh_{S_{L[\bx]}}} \kappa(s)[\bx]$.

We claim that this functor is exact and faithful. In order to prove this, notice that if we denote $p:S_{L[\bx]}\to S_L$ the projection morphism, we have that  $\D_{S_L/L}\subset p_*\D_{S_{L[\bx]}/L[\bx]}$, in particular $p_*E$ is a $\D_{S_L/L}$-module, and for every morphism $f$ of $\D_{S_{L[\bx]}/L[\bx]}$-modules, $p_*(f)$ is a morphism of $\D_{S_L/L}$-modules.

Let now consider a short exact sequence in $\D\Mod(S_{L[\bx]}/L[\bx])$
\[0\to \E'\to \E \to \E''\to 0\]
and let us apply the functor $\omega_{L[\bx]}$:
\[E'\otimes_{\Oh_{S_{L[\bx]}}} \kappa(s)[\bx]\to E\otimes_{\Oh_{S_{L[\bx]}}} \kappa(s)[\bx]\to E''\otimes_{\Oh_{S_{L[\bx]}}} \kappa(s)[\bx]\to 0\]
we need to check whether the sequence is exact on the left. Note that this is the case if and only if this holds for the same exact sequence considered as $L$-modules, that is after applying $p'_*$ to the sequence, where $p':\Spec L[\bx]\to \Spec L$ is given by the base change of $p$ along the section $s:\Spec L\to X_L$.  Note that the flat base change isomorphism respects the $\D$-module structure, in particular we have that after applying $p'_*$ the sequence is isomorphic to 
\[0\to\omega(p_*\E')\to \omega(p_*\E) \to \omega(p_*\E'')\to 0\]
which is exact by exactness of $\omega$.
In order to prove faithfulness, we use a similar argument: let $f:\E\to \E'$ be a $\D_{X[\bx]/L[\bx]}$-module map, if $\omega_L[\bx](f)$  (and hence by flat base change also $\omega(p_*(f))$) is zero then $p_*(f)$ must be zero by faithfulness of $\omega$, but then $f$ must be zero as well, as it is just $f$ seen as a morphism between $\Oh_X$-modules.

We want now to prove that every $\E\in\D\Mod(S_{L[\bx]}/L[\bx])$ is a quotient of a free object. Now, $p^*p_*\E$ is a locally free $\D\Mod(S_{L[\bx]}/L[\bx])$-module, and the unit map $p^*p_*\E\to\E$ is a surjective map of $\D$-modules.
Moreover if $\E$ is a subobject of $p^*\F$ for some $\F\in\D\Mod(S/K)$, then $p^*p_*\E$ is a subobject of $p^*p_*\F\simeq (p^*\F)^{p^e}$, in  particular  $\langle \F_{L[\bx]}\rangle_\otimes$ is spanned by its locally free objects.

Axiom $(C3)$  follows by the same arguments as in \ref{ex1} using Proposition~\ref{ins} as for axiom $(C2)$ and $(C4)$ they are obvious.
\end{proof}

\begin{corollary}\label{strat}
Let $S$ be a smooth geometrically connected scheme over $K$ with $S(K)\neq\emptyset$. Let $s\in S(K)$ and $L/K$ any field extension. Then for every $\E\in\D\Mod(S/K)$ the natural inclusion
\[\mu:\pi(\E_L,\omega_s)\subset \pi(\E,\omega_s)_L\]
is an isomorphism.
\end{corollary}
\begin{proof}
By Theorem~\ref{perfect} we can assume that $L$ is the algebraic closure of $K$, as both $\pi(\E_L,\omega_s)$ and $\pi(\E,\omega_s)$ are of finite type (see \cite[Prop.~2.20]{DM}) hence everything is defined over some finite field extension $L'$ of $K$, and taking the modular closure of $L'$ we are reduced to Proposition~\ref{dmod}.
\end{proof}

\section{Applications}\label{four}

\subsection{The structure of $\Repf(\pi(X_L))$}

Let $(\mathcal{T},\omega)$ and $(\mathcal{T}_L,\omega_L)$ be a geometric $(K,L)$-pair such that for $X\in\mathcal{T}$  we have $\pi(X_L)\simeq\pi(X)\otimes L$. Thanks to this isomorphism we can now describe $\langle X_L\rangle_\otimes$ in terms of $\langle X\rangle_\otimes$. 

By \cite[II,4.18]{Jan} we have that $\Hom_{\langle X\rangle_\otimes}(Y,Y')\otimes L=\Hom_{\langle X_L\rangle_\otimes}(Y_L,Y'_L)$ and  more in general $\Ext^i_{\langle X\rangle_\otimes}(Y,Y')\otimes L= \Ext^i_{\langle X_L\rangle_\otimes}(Y_L,Y'_L)$ for every $i$. Note that this means that if $\pi(X,\omega)$ is not semisimple we expect that the objects of $\langle X_L\rangle_\otimes$ are not all coming from $\langle X\rangle_\otimes$ by base change as new extension will appear as soon as $\Ext_{\langle X\rangle_\otimes}(Y,Y')\neq 0$.

If $K$ is algebraically closed we can be more precise in the comparison of $\langle X_L\rangle_\otimes$:

\begin{theorem}[{Jordan-H\"older, \cite[Thm~2.1]{Se}}] Let $\mathcal{C}$ be an abelian category and $X\in\mathcal{C}$ an object of finite length. Then $X$ admits a filtration, called \emph{composition series}
\[0=A_r\subset A_{r-1}\subset\dotsm\subset A_0=X\]
such that $S_i=A_i/A_{i+1}$ are simple for every $i$ (and different from zero). Moreover every such two filtration have the same length $r$ and their associated graded objects are isomorphic. The $S_i$ are called the \emph{composition factors} of $X$.
\end{theorem}

\begin{lemma}\label{ev}
Let $K$ be algebraically closed, then an object $Y$ in $\langle X\rangle_\otimes$ is simple if and only if $Y_L$ is simple in $\langle X_L\rangle_\otimes$.
\end{lemma}
\begin{proof}
One direction is trivial, for the converse it is enough to prove that if a representation $V$ of an algebraic $K$-group $G$ is simple then its base change to $L/K$ is simple as well.
Note that if $W\subset V_L$ is a (nontrivial) sub-representation, it is defined over some finite type $K$-algebra $A$, that is $W$ is defined as a sub-representation of $V_A$ of $G_A$.
In particular, for every closed (hence rational) point $z\in\Spec A$ we can evaluate $W$ in $z$ and get $ev_z(W\subset V_L)=ev_z(W)\to V$. All we need to show is then that $ev_z(W\to V_L)$ is not the zero morphism for some closed point $z\in\Spec A$. But this must be the case as otherwise $W\to V_L$ would also be the zero morphism.
\end{proof}

\begin{remark} If $K$ is not algebraically closed, $G$ is an affine group over $K$ and $\rho:G\to GL(V)$ is an irreducible representation, it is clearly not in general true that $\rho\otimes id:G_{\bar{K}}\to GL(V_{\bar{K}})$ is irreducible, an example being the subgroup of $\SL_{2,\R}$ of matrices of the form 
$\left(\begin{smallmatrix}
s&t\\-t&s
\end{smallmatrix}\right)$ acting on $\R^2$.
\end{remark}

\begin{corollary}
If $K$ is algebraically closed, simple objects in $\langle X_L\rangle_\otimes$ are of the form $X'_L$ for some $X'$ simple object in $\langle X\rangle_\otimes$.
\end{corollary}
\begin{proof}
If $0\subset X_r\subset\dotsm\subset X$ is a composition series for $X\in\mathcal{T}$ then by the previous lemma $0\subset(X_r)_L\subset\dotsm\subset X_L$ is a composition series for $X_L$. If $X'$ is a sub-object or a quotient of $X$ then its composition factors are isomorphic to a subset of the composition factors of $X$. This implies in particular that the simple objects in $\langle X_L\rangle_\otimes$ are the composition factors of $\mathcal{C}(X_L)$ for some construction of linear algebra $\mathcal{C}$, which completes the proof.
\end{proof}

\subsection{Positive characteristic $p$-curvature conjecture}

The following theorem, proved by Esnault and Langer and refined in \cite{bat:pcur} states (we omit the fiber functor when the ground field is algebraically closed as in this case the monodromy group is unique up to isomorphism):
\begin{theorem}[{\cite[Thm.~1.2]{EL:pcur},\cite[Thm.~4.3]{bat:pcur}}]\label{eldiv}
Let $K$ be an algebraically closed field of positive characteristic $p$. Let $X\to S$ be a smooth proper  morphism of $K$-varieties with geometrically connected fibers and let  $\E\in\D\Mod(X/S)$. Assume that there exists a dense subset $\tilde{S}\subset S(K)$ such that, for every $s\in \tilde{S}$, the stratified bundle $\E_s$ has finite monodromy and that the highest power of $p$ dividing $|\pi(\E_s)|$ is bounded over $\tilde{S}$. Let $\ebar$ be a geometric generic point of $S$, then
\begin{itemize}
\item[i)] there exists $f_\ebar:Y_\ebar\to X_\ebar$ a finite \'etale cover such that $f^*\E_\ebar$ decomposes as direct sum of stratified line bundles;
\item[ii)] if $K\neq \bar{\F}_p$ then $\pi(\E_\ebar)$ is finite.
\end{itemize} 
\end{theorem} 

We can apply the results in the previous section in order to get rid of any condition on the base field $K$, hence getting the following
\begin{theorem}\label{elgen} Let $K$ be any field. Let $X\to S$ be a smooth proper  morphism of $K$-varieties with geometrically connected fibers admitting a section $\sigma:S\to X$. Let $\E\in\D\Mod(X/S)$ and assume there exists a dense subset $\tilde{S}$ of rational points such that, for every $s\in \tilde{S}$, $\pi(\E_s,\omega_{\sigma(s)})$ is finite and that the highest power of $p$ dividing $|\pi(\E_s,\omega_{\sigma(s)})|$ is bounded over $\tilde{S}$. Then if $\ebar$ is a geometric generic point of $S$, $\pi(\E_\ebar)$ is finite.
\end{theorem} 

Using the results on the base change of the monodromy group, it is enough to prove the following
\begin{lemma}
Let $K$ be any field, let $S$ a scheme of finite type over $K$ and let $S'$ be a subset of the rational points which is dense in $S$. Then for every field extension $K\subset L$ the set $S'\subset S(L)=S_L(L)$ is dense in $S\otimes_KL$.
\end{lemma}
\begin{proof}
Let us first suppose that $L$ is separable over $K$, and by way of contradiction let us assume that $S'$ is contained in a proper closed sub-scheme $Z$ of $S_L$. Let $G=\Gal(L/K)$, acting on $S_L=S\otimes_KL$ as $\sigma\mapsto id\times \sigma$, then 
\[\bar{Z}=\bigcap_{\sigma\in G}\sigma(Z)\]
is a  proper closed sub-scheme of $S_L$ which is $G$-invariant and hence is defined over $K$, that is there exists $Z_0\subset S$ a proper closed sub-scheme such that $\bar{Z}=Z_0\otimes_KL$. But as $S'$ is a subset of the $K$-rational points on $S$, it is $\sigma$ invariant, in particular $S'\subset Z_0$ which gives a contradiction as $S'$ is dense in $S$.

We can hence suppose that $K$ is separably closed, in particular that $\overline{K}/K$ is purely inseparable. But then $\iota:\Spec\overline{K}\to \Spec K$ is a universal homeomorphism, in particular if $S'$ is dense in $S_L$ if and only if it is dense in $S$.

We are thus reduced to the case where $K$ is algebraically closed. Let $U$ be an open of $S_L$ such that $S'\cap U=\emptyset$ , then $U$ is defined over some finite type $K$-algebra $A$, that is there exists $\mathcal{U}$ a closed sub-scheme of $S_A$ over $\Spec A$ such that $U=\mathcal{U}\otimes_AL$. Moreover, $\mathcal{U}$ does not intersect the image of the sections $\Spec A\to U$ induced by the set $S'$.
Then there exists a closed point $s$ in $\Spec A$ such that the $S_A\otimes_A k(s)$ intersects $\mathcal{U}$ and thus $\mathcal{U}\otimes_A k(s)$ is an open of $S$ not intersecting $S'$, which gives a contradiction.
\end{proof}

Thanks to the previous lemma and to the base change of Corollary~\ref{strat}, we can, in the proof of Theorem \ref{elgen} first base change the base field $K$ to any field extension. As Theorem \ref{eldiv} holds for any $K$ algebraically closed and different from $\bar{\mathbb{F}}_p$, we are done.

\subsection{Base change of Tannakian fundamental groups}
If $(\mathcal{T},\omega)$ and $(\mathcal{T}_L,\omega_L)$ are a geometric $(K,L)$-pair, then there is a natural homomorphism of $L$-group schemes
\[\mu:\pi(\mathcal{T}_L,\omega_L)\to \pi(\mathcal{T},\omega)_L,\]
or, equivalently, a functor
\[M:\Repf_L(\pi(\mathcal{T},\omega)_L)\to\Repf_L(\pi(\mathcal{T}_L,\omega_L).\]

The homomorphism $\mu$ is not, in general, an isomorphism, see for example  \cite[VI, Ex.~1.2.6, b]{Sa:ct} and \cite[Cor.~23]{dS:fun} for counterexample with respect to the algebraic and the stratified fundamental groups.

Nevertheless using the base change of the monodromy of the single object one can prove the following:

\begin{lemma}\label{ff}
If $(\mathcal{T},\omega)$ and $(\mathcal{T}_L,\omega_L)$ are a $(K,L)$-geometric pair, such that for every $X$ 
\[\mu_X:\pi(X_L,\omega_L)\to\pi(X,\omega)_L\]
is an isomorphism, then the morphism $\mu$ is faithfully flat.
\end{lemma}
\begin{proof}
By \cite[Prop.~2.21]{DM}, $\mu$ is faithfully flat if and only if $M$ is fully faithful and for every $X\in  \Repf_L(\pi(\mathcal{T},\omega)_L)$ and every $X'\in \Repf_L(\pi(\mathcal{T}_L,\omega_L)$ sub-object of $M(X)$, there exists $Y\subset X$ such that $X'=M(Y)$.
But both properties can be checked at the level of monodromy groups, where they hold by assumption.
\end{proof}

It is hence natural to ask the following
\begin{question}
How can the kernel of $\mu$ be described, depending on $L/K$, for the various $(K,L)$-geometric pairs on which Lemma~\ref{ff} applies?
\end{question}

\subsection{The algebraic and stratified fundamental groups}
The results about the base change of the monodromy group can also be used to extend a fundamental result by dos Santos:
\begin{theorem}\label{smooth}
Let $K$ be any field of positive characteristic and let $S$ be a smooth $K$-variety. Then the monodromy group of every stratified bundle is smooth.
\end{theorem}
\begin{proof}
The result for $K$ algebraically closed is \cite[Cor.~12]{dS:fun}, and as smoothness is a geometric property the theorem follows from Corollary~\ref{strat}.
\end{proof}

Moreover in any characteristic the following holds:
\begin{proposition}
Let $L/K$ be an finite field extension and let $S$ be a geometrically connected smooth variety over $K$ and $s\in S(K)\neq\emptyset$, then the natural map
\[\mu:\pi(\D\Mod(S_L/L),\omega_s)\to \pi(\D\Mod(S/K),\omega_s)_L\]
is an isomorphism. The same holds for any algebraic extension if the characteristic of $K$ is zero.
\end{proposition}
\begin{proof}
Let us first assume that $L/K$ is finite, then one one side by Lemma~\ref{ff} $\mu$ is faithfully flat, on the other if $\F\in\D\Mod(S_L/L)$ then $\F$ is a direct factor of $p_*p^*\F$, where $p:S_L\to S$. Hence by \cite[Prop.~2.21]{DM} $\mu$ is a closed immersion and thus an isomorphism. If moreover $K$ has characteristic zero, then a $\D$-module is defined by finitely many data, in particular for a general algebraic extension $L/K$, every object in $\D\Mod(S_L/L)$ is defined over some $K'\subset L$ finite over $K$, and the same argument applies
\end{proof}

Indeed if the extension is algebraic but infinite, the previous result does not hold in general, even for very simple varieties:
\begin{lemma}
Let $L/K$ be an algebraic infinite field extension, then the natural faithfully flat map
\[\mu:\pi(\D\Mod(\A^1_L/L))\to\pi(\D\Mod(\A^1_K/K))_L\]
is never an isomorphism.
\end{lemma}
\begin{proof}
In order to show that $\mu$ cannot be a closed immersion, it is enough to prove that there exists $\E\in\D\Mod(\A^1_L/L)$ that is not a subquotient of $\E'_L$ for every $\E'\in\D\Mod(\A^1_K/K)$. We construct such $\E$ similarly as in \cite[Sec.~5]{bat:pcur} in the following way: first of all as $L/K$ is infinite, let us fix a countable family $(\alpha_i)_{i\in\N}$, such that they are all algebraically independent over $K$. Now consider $E$ to be the free two dimensional vector bundle on $\A^1_L$, with basis $e_1$ and $e_2$. 
Once fixed a parameter $x$, the ring of differential operators of $\A^1_L$ over $\Spec L$ is generated as a $\Oh_{\A^1_L}$-module, by $\de_x^{(k)}$ where $k$ runs along the natural numbers. The explicit relations between these generators are described in \cite[Cor.~2.5]{Ba:gen}. We define the action of $\D$ on $E$ to be given by $\de_x^{(k)}(e_1)=0$ for every $k>0$ and
\begin{equation}\label{exa}\de_x^{(k)}(e_2)=\begin{cases}
 \alpha_i \cdot e_1
 &\text{if } k=p^i, \\
\quad0&\text{else.}
\end{cases} \end{equation}
Exactly as in the proof of \cite[Lemma~5.2]{bat:pcur}, it is easy to check this actually respects the relations in the algebra $\D(\A^1_L/L)$ and hence defines a $\D$-module structure on $E$.
We claim that the resulting $\D$-module $\E$ is not subquotient of any object coming from $\D\Mod(\A^1_K/K)$. Indeed, assume there exists $\F\in\D\Mod(\A^1_K/K)$ be such that $\E\in\langle\F_L\rangle_\otimes$, then by Quillen--Suslin's theorem, the underlying module $F$ must be globally free, and fixing a basis the $\D(\A^1_K/K)$-action can be described by simply giving matrices $M_k$ with entries $m_{ij}^k$ given by $\de_x^{(k)}(e_i)=\sum_j m_{ij}^ke_j$ (see the discussion following \cite[Prop.~5.1]{bat:pcur} for a more exhaustive explanation).

In particular, $m_{ij}\in K$, now note that for every base change matrix $U\in GL_2(L[x])$, the matrices $M'_k$ in the new basis are given by
\begin{equation}\label{matrices}M'_k=\Big[\sum_{\substack {a+b=k\\a,b\geq 0 }}\de_x^{(a)}(U)M_b\Big]U^{-1}.
\end{equation}
Let us denote by $N_k$ the matrices denoting the $\D(\A^1_L/L)$-action on $F$, then if $F$ denotes the underlying vector bundle of $\F$, we can choose a basis of $F\otimes_KL$ such that every $N_k$ is the first $2\times 2$ minor of the $M'_k$. But this gives a contradiction: if $L'/K$ is the finite extension generated by the entries of the base change matrix $U$, then $\de_x^{(k)}(U)$ also have entries in $L'$ and hence $M'_k$ do too, but the $N_k$ do not, as the $\alpha_i$ are all algebraically independent over $K$.
\end{proof}

\addcontentsline{toc}{section}{\refname}
\printbibliography
\end{document}